\documentclass{amsart}
\usepackage{graphicx, amsmath, amssymb}
\usepackage{color}

\makeatletter \oddsidemargin.9375in \evensidemargin \oddsidemargin
\def\Corresponding author{$^{*}$\protect\footnotetext{$^{*}$ C\lowercase{orresponding author.}}}
\def\authorsaddresses#1{\dedicatory{#1}}

\newtheorem{thm}{Theorem}[section]
\theoremstyle {definition}
\newtheorem{cor}[thm]{Corollary}
\newtheorem{prop}[thm]{Proposition}

\newtheorem{lem}[thm]{Lemma}
\newtheorem{eg}[thm]{Example}

\numberwithin{equation}{section}

\begin{document}
\setcounter{page}{1}

\title[The annihilating-submodule graph of modules]{The annihilating-submodule graph of modules over commutative rings}

\author[H. Ansari-Toroghy and S. Habibi]{H. Ansari-Toroghy$^1$\Corresponding author and S. Habibi$^2$}

\authorsaddresses{$^1$ Department of pure Mathematics\\ Faculty of mathematical Sciences\\ University of Guilan,
P. O. Box 41335-19141, Rasht, Iran\\
e-mail: ansari@guilan.ac.ir\\
\vspace{0.5cm} $^2$ Department of pure Mathematics\\ Faculty of
mathematical Sciences\\ University of Guilan,
P. O. Box 41335-19141, Rasht, Iran\\
e-mail: sh.habibi@phd.guilan.ac.ir} 
\begin{abstract}
Let $M$ be a module over a commutative ring $R$. In this paper, we
continue our study of annihilating-submodule graph $AG(M)$ which
was introduced in (The Zariski topology-graph of modules over
commutative rings, Comm. Algebra., 42 (2014), 3283--3296). $AG(M)$
is a (undirected) graph in which a nonzero submodule $N$ of $M$ is
a vertex if and only if there exists a nonzero proper submodule
$K$ of $M$ such that $NK=(0)$, where $NK$, the product of $N$ and
$K$, is defined by $(N:M)(K:M)M$ and two distinct vertices $N$ and
$K$ are adjacent if and only if $NK=(0)$. We obtain useful
characterizations for those modules $M$ for which either $AG(M)$
is a complete (or star) graph or every vertex of $AG(M)$ is a
prime (or maximal) submodule of $M$. Moreover, we study coloring
of annihilating-submodule graphs.
\end{abstract}
\maketitle
\section{Introduction}
 Throughout this paper $R$ is a commutative ring with a non-zero
identity and $M$ is a unital $R$-module. By $N\leq M$ (resp. $N<
M$) we mean that $N$ is a submodule (resp. proper submodule) of
$M$. Let $\Lambda(M)$ and $\Lambda(M)^{*}$ be the set of proper
submodules of $M$ and nonzero proper submodules of $M$,
respectively.

Define $(N:_{R}M)$ or simply $(N:M)=\{r\in R|$ $rM\subseteq N\}$
for any $N\leq M$. We denote $((0):M)$ by $Ann_{R}(M)$ or simply
$Ann(M)$. $M$ is said to be faithful if $Ann(M)=(0)$.

Let $N, K\leq M$. Then the product of $N$ and $K$, denoted by
$NK$, is defined by $(N:M)(K:M)M$ (see \cite{af07}).

There are many papers on assigning graphs to rings or modules
(see, for example, \cite{al99, ah14, b88, br11}). The
annihilating-ideal graph $AG(R)$, was introduced and studied in
\cite{br11}. $AG(R)$ is a graph whose vertices are ideals of $R$
with nonzero annihilators and in which two vertices $I$ and $J$
are adjacent if and only if $IJ=(0)$.

In \cite{ah14}, we generalized the above idea to submodules of $M$
and defined the (undirected) graph $AG(M)$, called \textit {the
annihilating-submodule graph}, with vertices\\ $V(AG(M))$= $\{N
\leq M |$ there exists $(0)\neq K<M$ with $NK=(0)$\}. In this
graph, distinct vertices $N,L \in V(AG(M))$ are adjacent if and
only if $NL=(0)$. Let $AG(M)^{*}$ be the subgraph of $AG(M)$ with
vertices $V(AG(M)^{*})=\{ N<M$ with $(N:M)\neq Ann(M)|$ there
exists a submodule $K<M$ with $(K:M)\neq Ann(M)$ and $NK=(0)\}$.
Note that $M$ is a vertex of $AG(M)$ if and only if there exists a
nonzero proper submodule $N$ of $M$ with $(N:M)=Ann(M)$ if and
only if every nonzero submodule of $M$ is a vertex of $AG(M)$.

A prime submodule of $M$ is a submodule $P\neq M$ such that
whenever $re\in P$ for some
  $r\in R$ and $e \in M$, we have $r\in (P:M)$ or $e\in P$ \cite{lu84, mm97}.

The prime spectrum (or simply, the spectrum) of $M$ is the set of
all prime submodules of $M$ and denoted by $Spec(M)$. Also,
$Max(M)$ will denote the set of all maximal submodules of $M$.

The prime radical $rad_{M}(N)$ is defined to be the intersection
of all prime submodules of $M$ containing $N$, and in case $N$ is
not contained in any prime submodule, $rad_{M}(N)$ is defined to
be $M$ \cite{lu84}.

Let $Z(R)$ and $Nil(R)$ be the set of zero-divisors and nilpotent
elements of $R$, respectively. Let $Z_{R}(M)$ or simply $Z(M)$ be
the set $\{r\in R|$ $rm=0$ for some $0\neq m\in M \}$.

 Let $N$ and $K$ be submodules of $M$. Then the product of $N$ and $K$
 is defined by $(N:M)(K:M)M$ and denoted
 by $NK$ (see \cite{af07}).

A clique of a graph is a maximal complete subgraph and the number
of vertices in the largest clique of graph $G$, denoted by
$cl(G)$, is called the clique number of $G$. Let $\chi(G)$ denote
the chromatic number of the graph $G$, that is, the minimal number
of colors needed to color the vertices of $G$ so that no two
adjacent vertices have the same color. Obviously $\chi(G)\geq
cl(G)$.

In section 2, we continue all modules $M$ for which $AG(M)$ is a
complete (resp. star) graph or every vertex of $AG(M)$ is a prime
(or maximal) submodule (see Theorems \ref{t2.14}, \ref{t2.15}, and
\ref{t2.17}). In section 3, we study the coloring of the
annihilating-submodule graph of modules. At first, among other
results, we give a characterization of $\chi(AG(M)^{*})=2$ (see
Theorem \ref{t3.2}). It is shown that for a semiprime module $M$,
the following conditions are equivalent. (1) $\chi(AG(M)^{*})$ is
finite. (2) $cl(AG(M)^{*})$ is finite. (3) $AG(M)^{*}$ does not
have an infinite clique (see Corollary \ref{c3.8}). Also, it is
shown that for a faithful module $M$ with $rad_{M}(0)=(0)$, the
following conditions are equivalent. (1) $\chi(AG(M)^{*})$ is
finite. (2) $cl(AG(M)^{*})$ is finite. (3) $AG(M)^{*}$ does not
have an infinite clique. (4) $R$ has a finite number of prime
ideals (see Proposition \ref{p3.11}).

Let us introduce some graphical notions and denotations that are
used in what follows:
A graph $G$ is an ordered triple $(V(G), E(G), \psi_G )$ consisting of a
nonempty set of vertices,
 $V(G)$, a set $E(G)$ of edges, and an incident function $\psi_G$ that associates an
 unordered pair
 of distinct vertices with each edge. The edge $e$ joins $x$ and $y$ if $\psi_G(e)=\{x, y\}$, and we
 say $x$ and $y$ are adjacent. A path in graph $G$ is a finite sequence of vertices $\{x_0,
x_1,\ldots ,x_n\}$, where $x_{i-1}$ and $x_i$ are adjacent for
each $1\leq i\leq n$ and we denote $x_{i-1} - x_i$ for existing an
edge between
 $x_{i-1}$ and $x_i$.

A graph $H$ is a subgraph of $G$ if $V(H)\subset V(G)$,
$E(H)\subseteq E(G)$ and $\psi_H$ is the restriction of $\psi_G$
to $E(H)$. A bipartite graph is a graph whose vertices can be
divided into two disjoint sets $U$ and $V$ such that every edge
connects a vertex in $U$ to one in $V$; that is, $U$ and $V$ are
each independent sets and complete bipartite graph on $n$ and $m$
vertices, denoted by $K_{n, m}$, where $V$ and $U$ are of size $n$
and $m$, respectively, and $E(G)$ connects every vertex in $V$
with all vertices in $U$. Note that a graph $K_{1, m}$ is called a
star graph and the vertex in the singleton partition is called the
center of the graph.
 (see \cite{r05}).

\section{The Annihilating-submodule graph}

An ideal $I\leq R$ is said to be nil if $I$ consist of nilpotent
elements; $I$ is said to be nilpotent if $I^{n}=(0)$ for some
natural number $n$.

\begin{prop}\label{p2.1} Suppose that $e$ is an idempotent element of
$R$. We have the following statements.

\begin {itemize}
\item [(a)] $R=R_{1}\oplus R_{2}$, where $R_{1}=eR$ and
$R_{2}=(1-e)R$. \item [(b)] $M=M_{1}\oplus M_{2}$, where
$M_{1}=eM$ and $M_{2}=(1-e)M$. \item [(c)] For every submodule $N$
of $M$, $N=N_{1}\times N_{2}$ such that $N_{1}$ is an
$R_{1}$-submodule $M_{1}$, $N_{2}$ is an $R_{2}$-submodule
$M_{2}$, and $(N:_{R}M)=(N_{1}:_{R_{1}}M_{1})\times
(N_{2}:_{R_{2}}M_{2})$.  \item [(d)] For submodules $N$ and $K$ of
$M$, $NK=N_{1}K_{1} \times N_{2}K_{2}$ such that $N=N_{1}\times
N_{2}$ and $K=K_{1}\times K_{2}$.
\end{itemize}

\end{prop}

\begin{proof}
This is clear.
\end{proof}

We need the following lemmas.

\begin{lem}\label{l2.2} (See \cite[Proposition 7.6]{af74}.)
Let $R_{1}, R_{2}, \ldots , R_{n}$ be non-zero ideals of $R$. Then
the following statements are equivalent:

\begin{itemize}
\item [(a)] $_{R}R= R_{1} \oplus \ldots \oplus R_{n}$; \item [(b)]
As an abelian group $R$ is the direct sum of $ R_{1}, \ldots ,
R_{n}$; \item [(c)] There exist pairwise orthogonal central
idempotents $e_{1},\ldots, e_{n}$ with $1=e_{1}+ \ldots +e_{n}$,
and $R_{i}=Re_{i}$, $i=1, \ldots ,n$.
\end{itemize}

\end{lem}

\begin{lem}\label{l2.3} (See \cite[Theorem 21.28]{l91}.)
Let $I$ be a nil ideal in $R$ and $u\in R$ be such that
 $u+I$ is an idempotent in $R/I$. Then there exists an idempotent
 $e$ in $uR$ such that $e-u\in I$.
\end{lem}

\begin{lem}\label{l2.4}
Let $N$ be a minimal submodule of $M$ and let $Ann(M)$ be a nil
ideal. Then we have $N^{2}=(0)$ or $N=eM$ for some idempotent
$e\in R$.
\end{lem}

\begin{proof}
Assume $N^{2}\neq (0)$. Since $N^{2}\neq (0)$ and $N$ is a minimal
submodule of $M$, we have $(N:M)m\neq (0)$ for some $m\in N$ so
that $(N:M)m=N$. Choose $u\in (N:M)$ such that $m=um$. So $N=uM$.
Since $m\in ((0):_{N}u-1)$, $N= ((0):_{N}u-1)$ and hence
$u(u-1)M=(0)$. Thus $u(u-1)\in Ann(M)$. By Lemma \ref{l2.3}, there
is an idempotent $e$ in $R$ with $e-u\in Ann(M)$. So $(e-u)M=(0)$.
It is clear that $eM=uM$. Hence $N=eM$.
\end{proof}

\begin{thm}\label{t2.5} Let $Ann(M)$ be a nil ideal.
There exists a vertex of $AG(M)$ which is adjacent to every other
vertex if and only if $M=eM\oplus (1-e)M$, where $eM$ is a simple
module and $(1-e)M$ is a prime module for some idempotent $e\in R$
or $Z(M)=Ann((N:M)M)$, where $N$ is a nonzero proper submodule of
$M$ or $M$ is a vertex of $AG(M)$.
\end{thm}

\begin{proof}
Suppose that $N$ is adjacent to every other vertex of $AG(M)$,
$Z(M)\neq Ann((K:M)M)$ for every nonzero proper submodule $K$ of
$M$ and $M$ is not a vertex of $AG(M)$. If $N^{2}=(0)$, then
$Z(M)=Ann((N:M)M)$, a contradiction (note that if $r\in Z(M)$,
then there exists a nonzero element $m\in M$ such that $rm=0$. If
$rM=(0)$, then $r\in Ann((N:M)M)$. Otherwise, since
$(rM:M)(mR:M)M=(0)$, we have $(rM:M)(N:M)M=(0)$). Thus $N^{2}\neq
(0)$. Again by the above arguments, $N=(N:M)M$. By Lemma
\ref{l2.4}, $N=eM$ for some idempotent $e\in R$. We may assume
that $R=R_{1}\times R_{2}$ and $M=M_{1}\times M_{2}$. Also, by
Proposition \ref{p2.1}, we may assume that $M_{1}\times (0)$ is
adjacent to every other vertex of $AG(M)$. Now we show that
$M_{2}$ is a prime module. Otherwise, there exist $0\neq m\in
M_{2}$ and $r\in R$ such that $rm=0$ and $r\notin Ann(M_{2})$. It
follows that $(M_{1}\times mR_{2})((0)\times rM_{2})=(0)$. So
$M_{1}\times (0)$ is adjacent to $M_{1}\times mR_{2}$. This
implies that $R_{1}=(0)$, a contradiction. Therefore $M_{2}$ is a
prime module. Conversely, assume that $M=eM\oplus (1-e)M$, where
$eM$ is a simple module and $(1-e)M$ is a prime module such that
$e$ is an idempotent. One can see that $eM\times (0)$ is adjacent
to every other vertex of $AG(M)$. If $M$ is a vertex of $AG(M)$,
then there exists a nonzero proper submodule $N$ of $M$ such that
$(N:M)=Ann(M)$ and hence $N$ is adjacent to every other vertex.
Now suppose that $Z(M)=Ann((N:M)M)$, where $N$ is a nonzero proper
submodule of $M$. Then it is easy to see that $N$ is a vertex of
$AG(M)$ which is adjacent to every other vertex or $M$ is a vertex
of $AG(M)$.
\end{proof}

\begin{eg}\label{e2.6} Let $M:=\Bbb Z_{2}\oplus \Bbb Z_{3}$ as a
 $\Bbb Z_{12}-$module. Clearly, $Ann(M)=\{\bar{0}, \bar{6}\}$ is a
 nil ideal and $AG(\Bbb Z_{2}\oplus \Bbb
Z_{3})$ is a star graph with the only edge $\Bbb Z_{2}\oplus (0) -
(0)\oplus \Bbb Z_{3}$.
\end{eg}

\begin{thm}\label{t2.7} Let $M$ be a faithful module.
There exists a vertex of $AG(M)^{*}$ which is adjacent to every
other vertex of $AG(M)^{*}$ if and only if $M=M_{1}\oplus M_{2}$,
where $M_{1}$ is a simple module and $M_{2}$ is a prime module or
$Z(R)$ is an annihilator ideal.
\end{thm}

\begin{proof}
$( \Longrightarrow )$. Suppose that $Z(R)$ ia not an annihilator
ideal and $N$ is adjacent to every other vertex. If $N^{2}=(0)$,
then $Z(R)=Ann((N:M))$, a contradiction (note that if $r\in Z(R)$,
then there exists a nonzero element $s\in R$ such that $rs=0$. So
we have $(rM:M)(sM:M)=(0)$. Hence $(rM:M)(N:M)=(0)$). Thus
$N^{2}\neq (0)$. Now the claim follows by using similar arguments
as in the proof
of theorem \ref{t2.5}.\\
$( \Longleftarrow )$. Assume that $M=M_{1}\oplus M_{2}$, where
$M_{1}$ is a simple $R$-module and $M_{2}$ is a prime $R$-module.
Since $M_{1}$ is a simple $R$-module, $Ann(M_{1})$ is a maximal
ideal of $R$ and since $Ann(M)=(0)$, we have
$Ann(M_{2})+Ann(M_{1})=R$ and so we may assume that $R=R_{1}\oplus
R_{2}$. Then Lemma \ref{l2.2} and Theorem \ref{t2.5} imply that
there is a vertex of $AG(M)^{*}$ which is adjacent to every other
vertex of $AG(M)^{*}$. Now let $Z(R)=Ann(I)$ for some nonzero
proper ideal $I$ of $R$. In this case, clearly, $IM$ is a vertex
of $AG(M)^{*}$ which is adjacent to every other vertex of
$AG(M)^{*}$.
\end{proof}

\begin{eg}\label{e2.8} It is easy to see that $\Bbb Q\oplus \Bbb Q$ as $\Bbb
Q\oplus \Bbb Z-$module is faithful and $AG(\Bbb Q\oplus \Bbb
Q)^{*}$ is a star graph with the only edge $\Bbb Q\oplus (0) -
(0)\oplus \Bbb Q$.
\end{eg}

\begin{cor}\label{c2.9} Let $R$ be a reduced ring and let $Ann(M)$ be a nil ideal.
Then the following statements are equivalent.

\begin {itemize}
\item [(a)] There is a vertex of $AG(M)^{*}$ which is adjacent to
every other vertex of $AG(M)^{*}$. \item [(b)] $AG(M)^{*}$ is a
star graph. \item [(c)] $M=M_{1}\oplus M_{2}$, where $M_{1}$ is a
simple module and $M_{2}$ is a prime module.
\end{itemize}

\end{cor}

\begin{proof}
$(a) \Rightarrow (b)$ Suppose that there is a vertex of
$AG(M)^{*}$ which is adjacent to every other vertex. If
$Z(R)=Ann(x)$ for some $0\neq x\in R$, then we have $x^{2}=0$, a
contradiction. Therefore by Theorem \ref{t2.7}, $M=M_{1}\oplus
M_{2}$, where $M_{1}$ is a simple module and $M_{2}$ is a prime
module. Then every nonzero submodule of $M$ is of the form
$M_{1}\times N_{2}$ and $(0)\times N_{2}$, where $N_{2}$ is a
nonzero submodule of $M_{2}$. By our hypothesis, we can not have
any vertex of the form $M_{1}\times N_{2}$, where $N_{2}$ is a
nonzero proper submodule of $M_{2}$. Also $M_{1}\times (0)$ is
adjacent to every other vertex, and non of the submodules of the
form $(0)\times N_{2}$ can be adjacent to each other. So
$AG(M)^{*}$ is
a star graph.\\
$(b) \Rightarrow (c)$ This follows by Theorem \ref{t2.7}.\\
$(c)\Rightarrow (a)$ This follows by Theorem \ref{t2.7}.
\end{proof}

Let $M$ be an $R$-module. The set of associated prime ideals of
$M$, denoted by $Ass_{R}(M)$ (or simply $Ass(M)$), is defined as
$Ass(M)=\{p\in Spec(R)|$ $p=(0:_{R}m)$ for some $0\neq m\in M \}$.

\begin{cor}\label{c2.10} Let $R$ be an Artinian ring and let $Ann(M)$ be a nil ideal.
 Then there is a vertex of $AG(M)$ which is adjacent to every other vertex
  if and only if either $M=M_{1}\oplus M_{2}$, where
$M_{1}$ is a simple module and $M_{2}$ is a prime semisimple
module or $R$ is a local ring with maximal ideal $p\in Ass(M)$ or
$M$ is a vertex of $AG(M)$.
\end{cor}

\begin{proof}
$(\Longrightarrow)$ Let $N$ be a vertex of $AG(M)$ which is
adjacent to every other vertex and suppose $M$ is not a vertex of
$AG(M)$. As we have seen in Theorem \ref{t2.5}, either
$M=M_{1}\oplus M_{2}$, where $M_{1}$ is a simple module and
$M_{2}$ is a prime module or $Z(M)=Ann((K:M)M)$, where $K$ is a
nonzero proper submodule of $M$. Let $M=M_{1}\oplus M_{2}$, where
$M_{1}$ is a simple module and $M_{2}$ is a prime module. It is
easy to see that $M_{2}$ is a vector space over $R/Ann(M_{2})$ and
so is a semisimple $R$-module. If $Z(M)$ is an ideal of $R$
, since $R$ is an Artinian ring, then $Z(M) = p\in Ass(M)$.\\
$(\Longleftarrow)$ First suppose that $R$ is not a local ring.
Hence by \cite[Theorem 8.7]{ati69}, $R = R_{1}\times \ldots \times
R_{n}$, where $R_{i}$ is an Artinian local ring for $i=1, \ldots ,
n$. By Lemma \ref{l2.2} and Theorem \ref{t2.5}, we may assume that
$eM\times (0)$ is adjacent to every other vertex of $AG(M)$. If
$R$ is a local ring with maximal ideal $p\in Ass(M)$, then there
exists $0\neq m\in M$ such that $p=Ann(m)$. We claim that $Rm$ is
adjacent to every other vertex. Suppose $N$ is a vertex. We have
$(N:M)\subseteq p$. Hence $(N:M)(mR)=(0)$. So we have $N(mR)=(0)$.
\end{proof}

\begin{eg}\label{e2.11} Let $M:=\Bbb Z_{3}\oplus \Bbb Z_{8}$ as a
 $\Bbb Z_{48}-$module. Clearly, $Ann(M)=\{\bar{0}, \bar{24}\}$ is a
 nil ideal and $AG(\Bbb Z_{3}\oplus \Bbb
Z_{8})$ is a star graph with the center $\Bbb Z_{3}\oplus (0)$ and
$V(AG(\Bbb Z_{3}\oplus \Bbb Z_{8}))=\{\Bbb Z_{3}\oplus (0),
(0)\oplus \Bbb Z_{8}, (0)\oplus N, (0)\oplus K \}$, where
$N=(\bar{0},\bar{2})\Bbb Z_{48}$ and $K=(\bar{0},\bar{4})\Bbb
Z_{48}$.
\end{eg}

\begin{cor}\label{c2.12} Let $R$ be an Artinian ring and let $M$ be a faithful $R$-module.
 Then there is a vertex of $AG(M)^{*}$ which is adjacent to every other vertex
if and only if either $M=M_{1}\oplus M_{2}$, where $M_{1}$ and
$M_{2}$ are simple modules or $R$ is a local ring with maximal
ideal $p\in Ass(M)$.
\end{cor}

\begin{proof}
$(\Longrightarrow)$ Let $N$ be a vertex of $AG(M)^{*}$ which is
adjacent to every other vertex. So there is a vertex of $AG(R)$
which is adjacent to every other vertex of $AG(R)$. By
\cite[Corollary 2.4]{br11}, we may assume that $R= F_{1}\oplus
F_{2}$, where $F_{1}$ and $F_{2}$ are fields or $Z(R)$ is an
annihilator ideal. If $R= F_{1}\oplus F_{2}$, then $AG(R)$ is a
complete graph. It follows that $AG(M)^{*}$ is a complete graph.
Hence $AG(M)^{*}$ have exactly two vertices $M_{1}\times (0)$ and
$(0)\times M_{2}$. So $M_{2}$ is a simple module. If $Z(R)$ is an
annihilator ideal, since $R$ is an Artinian ring, then $Z(R)$ is
the unique maximal ideal of $R$. Since $R$ is a Noetherian ring,
$Z(M) = p\in Ass(M)$ and hence $Z(R)=Z(M)$.\\ $(\Longleftarrow)$
This is clear by Corollary \ref{c2.10}.
\end{proof}

\begin{lem}\label{l2.13} Let $R$ be an Artinian ring and assume
that $Ann(M)$ is a nil ideal and $AG(M)$ is a star graph. Then
either $M=M_{1}\oplus M_{2}$, where $M_{1}$ is a simple module and
$M_{2}$ is a prime semisimple module or $R$ is a local ring with
maximal ideal $p=Ann(m)$, $(mR)^{2}=(0)$ and $p^{4}M=(0)$ or $M$
is a vertex of $AG(M)$.

\end{lem}

\begin{proof}
Let $R$ be an Artinian ring and assume that
 $AG(M)$ is a star graph and $M$ is not a vertex of $AG(M)$.
  Then by Corollary \ref{c2.10},
either $M=M_{1}\oplus M_{2}$, where $M_{1}$ is a simple module and
$M_{2}$ is a prime semisimple module or $R$ is a local ring with
maximal ideal $p\in Ass(M)$. If $M=M_{1}\oplus M_{2}$, then there
is nothing to prove. If $R$ is a local ring with maximal ideal
$p=Ann(m)$, where $0\neq m\in M$, then as we showed in Corollary
\ref{c2.10}, $mR$ is adjacent to every other vertex. Since $R$ is
an Artinian ring and $M$ is not a vertex, there exists an integer
$n> 1$ such that $p^{n}M=(0)$ and $p^{n-1}M\neq (0)$. As $AG(M)$
is a star graph (resp. $(mR:M)\subseteq p$), we have $p^{4}M=(0)$
and $mR=p^{3}M$ (resp. $(mR)^{2}=(0)$).
\end{proof}

\begin{thm}\label{t2.14} Let $R$ be an Artinian ring and assume that $M$
is not a vertex of $AG(M)$ and $Ann(M)$ is a nil ideal. Then
 $AG(M)$ is a star graph if and only if either $M=M_{1}\oplus M_{2}$, where
 $M_{1}$ and $M_{2}$ are simple modules or $R$ is a local ring with maximal ideal $p=(0:m)\in
Ass(M)$ and one of the following cases holds.

\begin {itemize}
\item [(a)] $p^{2}M=(0)$ and $pM=mR$ is the only nonzero proper
submodule of $M$. \item [(b)] $p^{3}M=(0)$ and $p^{2}M=mR$ is the
only minimal submodule of $M$ and for every distinct proper
submodules $N_{1}, N_{2}$ of $M$ such that $mR\neq N_{i}$ $(i=1,
2)$, $N_{1}N_{2}=Rm$. \item [(c)] $p^{4}M=(0)$, $p^{3}M\neq (0)$,
and $\Lambda(M)^{*}=\{N< M|$ $(N:M)=(pM:M)\}\cup \{p^{2}M,
p^{3}M=mR\}$.
\end{itemize}

\end{thm}

\begin{proof}
$(\Longrightarrow)$ Suppose that $AG(M)$ is a star graph. By Lemma
\ref{l2.13}, we may have $p^{4}M=(0)$. We proceed
by the following cases:\\\\
Case 1. $p^{2}M=(0)$. Hence every nonzero proper submodule $N$ of
$M$ is a vertex and $N^{2}=(0)$. It is clear that $pM$ is a
$p$-prime submodule of $M$ and is adjacent to every other vertex.
Thus $pM=Rm$. Since for every nonzero proper submodule $N$ and $K$
of $M$, $NK=(0)$ and $AG(M)$ is a star graph, $M$ has at most two
nonzero proper submodules. So $M$ is a Noetherian module and $Rm$
is a subset of every nonzero submodule of $M$. It is easy to see
that $M$ is cyclic and hence a multiplication module. It follows
that $pM=mR$ is the
only nonzero proper submodule of $M$.\\\\
Case 2. $p^{3}M=(0)$ and $p^{2}M\neq (0)$. It is clear that
$p^{2}M$ is adjacent to every other vertex. So $p^{2}M=Rm$. We
claim that $Rm$ is the only minimal submodule of $M$. Suppose $N$
is another minimal submodule of $M$. It is easy to see that
$Ann(N)=p$. Let $K$ be a nonzero proper submodule of $M$. Thus
$(K:M)\subseteq Ann(N)=p$ so that $NK=(0)$. Hence $N=Rm$. Finally,
suppose $N_{1}, N_{2}\neq Rm$ are distinct nonzero proper
submodules of $M$. We have $(N_{1}:M), (N_{2}:M)\subseteq p$.
Since $AG(M)$ is a star graph, we have $N_{1}N_{2}\neq (0)$. Hence
by minimality of $p^{2}M=mR$, $N_{1}N_{2}=mR$.\\\\
Case 3. $p^{4}M=(0)$ and $p^{3}M\neq (0)$. Since $AG(M)$ is a star
graph and $R$ is a local ring, then the center of the star graph
must be a nonzero cyclic submodule $p^{3}M=mR$. Since
$p^{3}\subsetneq p^{2}$, there are elements $a, b\in p\setminus
p^{2}$ such that $ab\in p^{2}\setminus p^{3}$. Then $abp^{2}M =
(0)$, so $p^{2}M = abM$ because $AG(M)$ is a star graph. If
$a^{2}\in p^{3}M$, then $(aM)(abM) = (0)$, a contradiction with
the star shaped assumption. Hence $p^{2}M = a^{2}M$. Another
application of the star shaped assumption yields $a^{3}\neq 0$. So
$p^{3}M = a^{3}M$. For $c\in p\setminus p^{2}$, $ca^{2} \in
a^{3}R\setminus (0)$. We conclude that $c \in aR$. Hence $p = aR$
and hence every non-zero ideal of $R$ is a power of $p$. So
$(aM:M)=aR$, $(a^{2}M:M)=a^{2}R$, and $(a^{3}M:M)=a^{3}R$. It is
easy to see that $AG(M)$ is a star graph with the center
$a^{3}M=mR$, and the other vertices are $aM$, $a^{2}M$, and every
nonzero proper submodule $N$ of $M$ with $(N:M)=aR$. (Note that
$Spec(M)=\{(0)\neq N<M|$ $N\neq a^{2}M, a^{3}M\}$.) \\
$(\Longleftarrow)$ This is clear.
\end{proof}

\begin{thm}\label{t2.15} Assume that $M$ is a faithful module and is
not a vertex of
$AG(M)$.
Then $AG(M)$ is a complete graph if and only if $M$ is one of the
three types of modules.

\begin {itemize}
\item [(a)] $M=M_{1}\oplus M_{2}$, where $M_{1}$ and $M_{2}$ are
simple modules, \item [(b)] $Z(R)$ is an ideal with
$(Z(R))^{2}=(0)$, or \item [(c)] Every nonzero proper submodule of
$M$ is a vertex, $Spec(M)=Max(M)=\{aM\}$, where $R$ is a local
ring with exactly two nonzero proper ideals $Z(R)=aR$, $Z(R)^{2}$
such that $a^{3}=0$, and for every nonzero proper submodule $N\neq
aM$, $(N:M)=a^{2}R$.
\end{itemize}

\end{thm}

\begin{proof}
$(\Longrightarrow)$ Assume that $AG(M)$ is a complete graph. So
$AG(R)$ is a complete graph. By Theorem \ref{t2.5}, $M=M_{1}\oplus
M_{2}$, where $M_{1}$ is a simple module and $M_{2}$ is a prime
module or $Z(R)$ is an ideal. Suppose that we have the first case.
If $M_{2}$ has a nonzero proper submodule, say $N_{2}$, then
$(0)\times M_{2}$ and $(0)\times N_{2}$ are vertices of $AG(M)$
which are not adjacent, a contradiction. Thus $M_{2}$ can not have
any nonzero proper submodule, and hence it is a simple module. Now
assume that $Z(R)$ is an ideal of $R$. So (b) holds if
$Z(R)^{2}=(0)$. Otherwise, then by \cite[Theorem 2.7]{br11}, $R$
is a local ring with exactly two nonzero proper ideals $Z(R)=aR$
and $Z(R)^{2}$ (note that $a^{3}=0$). Hence every nonzero proper
submodule $M$ is a vertex. Since $AG(M)$ is a complete graph, for
every nonzero proper submodule $N\neq aM$, we
have $(N:M)=a^{2}R$. It follows that $Spec(M)=Max(M)=\{aM\}$. \\
$(\Longleftarrow)$ This is clear.
\end{proof}

\begin{cor}\label{c2.16} Assume that $M$ is a faithful module and is
not a vertex of $AG(M)$. Then we have the following.

\begin {itemize}
\item [(a)] $AG(M)$ is a complete graph with one vertex if and
only if $M$ has only one nonzero proper submodule. \item [(b)]
$AG(M)$ is a graph with two vertices if and only if $M=M_{1}\oplus
M_{2}$, where $M_{1}$ and $M_{2}$ are simple modules or $M$ is a
module with exactly two nonzero proper submodules $Z(R)M$ and
$Z(R)^{2}M$. \item [(c)] $AG(M)$ is a graph with three vertices if
and only if $M$ has exactly three nonzero proper submodules
$m_{1}R, m_{2}R, m_{3}R$ such that $m_{3}R=m_{1}R \cap m_{2}R,
Z(R)=Ann(m_{3}), (m_{1}R)^{2} = (m_{2}R)^{2} = (m_{3}R)^{2}=(0)$,
where $0\neq m_{1}, m_{2}, m_{3}\in M$, or $\Lambda(M)^{*}=
\{Z(R)M, Z(R)^{2}M, Z(R)^{3}M\}$.
\end{itemize}

\end{cor}

\begin{proof}
(a) $(\Longrightarrow)$ This follows by \cite[Theorem 3.6 and
Proposition 3.5]{ah14}.\\ $(\Longleftarrow)$ Suppose $M$ has only
one nonzero
 proper submodule. It follows that $M\cong R$ and hence
 \cite[Corollary 2.9(a)]{br11} completes the proof.\\
(b) $(\Longrightarrow)$ Suppose $AG(M)$ is a graph with two
vertices. By \cite[Theorem 3.6 and Proposition 3.5]{ah14}, $M$ has
exactly two nonzero proper submodules. Since $AG(M)$ is connected,
then $AG(M)$ is a complete (or star) graph. Thus by Theorem
\ref{t2.15} and Theorem \ref{t2.14}, $M=M_{1}\oplus M_{2}$, where
$M_{1}$ and $M_{2}$ are simple modules or $M$ is a module with
exactly two nonzero proper submodules $Z(R)M$
and $Z(R)^{2}M$. \\ $(\Longleftarrow)$ This is clear.\\
(c) $(\Longrightarrow)$ Suppose $AG(M)$ is a graph with three
vertices. By \cite[Theorem 3.6 and Proposition 3.5]{ah14}, $M$ has
exactly three nonzero proper submodules. Since $AG(M)$ is
connected, either it is a complete (or star) graph. If $AG(M)$ is
a complete graph, then we may have the cases (b) and (c) in
Theorem \ref{t2.15}. First we assume that the case (b) in this
theorem is true. Then $Z(R)$ is an ideal of $R$ with
$Z(R)^{2}=(0)$. It follows that $R$ is an Artinian ring and
$Z(R)=Nil(R)=Ann(m)$ is the only prime ideal of $R$, where $Rm$ is
a minimal submodule of $M$. Let $N_{1}$, $N_{2}$, and $N_{3}$ be
the only nonzero proper submodules
of $M$. We proceed by the following cases:\\\\
Case 1. $N_{1}\subset N_{2}\subset N_{3}$. Then we have $M$ and
$R$ are
isomorphic which is a contradiction by \cite[Corollary
2.9]{br11}.\\\\
Case 2. $N_{1}$, $N_{2}$, and $N_{3}$ are minimal submodules of
$M$. In this case, $M$ is a multiplication cyclic module. But this
yields a contradiction.\\\\
Case 3. $N_{1}\subset N_{2}$, and $N_{3}$ is not comparable with
$N_{i}$, $i=1,2$. Then since $(N_{2}:M)=(N_{3}:M)=Z(R)$, it
follows
that $N_{1}= N_{3}$ or $N_{2}= N_{3}$, a contradiction.\\\\
Case 4. $N_{1}\subset N_{2}$, and $N_{3}\subset N_{2}$. Then we
have $(N_{2}:M)=Z(R)$. If $M$ is a multiplication module, then it
is cyclic and hence similar to the case (1), we get a
contradiction. Otherwise, $N_{1}\subset N_{3}$ or $N_{3}\subset
N_{1}$,
which is again a contradiction.\\\\
Case 5. $N_{3}\subset N_{1}$ and $N_{3}\subset N_{2}$. It follows
that $N_{1}=Rm_{1}$, $N_{2}=Rm_{2}$, $N_{3}=Rm_{3}=N_{1}\cap
N_{2}$, $Z(R)=Ann(m_{3})$, and $N^{2}_{1}
=N^{2}_{2}=N^{2}_{3}=(0)$, as desired.

Now suppose that the case (c) in Theorem \ref{t2.15} is true. Then
we have $\Lambda(M)^{*}= \{aM, a^{2}M, N\}$ such that
$Spec(M)=Max(M)=\{aM\}$, where $Z(R)=aR$ and $(N:M)=a^{2}R$. It
follows that $N\subseteq aM$ and hence similar to case (4), we get
again a contradiction. Finally suppose that $AG(M)$ is a star
graph with three vertices. Then we may have the cases (b) and (c)
in Theorem \ref{t2.15}. We prove that the case (b) in this theorem
is not true. Otherwise, $Z(R)^{3}M=(0)$ and $Z(R)^{2}M$ is the
only minimal submodule of $M$. Let $N$ be a nonzero proper
submodule of $M$ and $N\neq pM, p^{2}M$. If $N$ is a maximal
submodule of $M$, then $p^{2}M\subset pM\subset N$ and hence
similar to case (1), we have a contradiction. Otherwise,
$p^{2}M\subset N\subset pM$, which is again a contradiction. Thus
$M$ has exactly three nonzero proper
submodules $Z(R)M$, $Z(R)^{2}M$, and $Z(R)^{3}M$.\\
$(\Longleftarrow)$ This is clear.

\end{proof}

\begin{thm}\label{t2.17} Suppose that $R$ is not a domain and $M$ is a
 faithful module. If every vertex of $AG(M)$ is
a prime submodule of $M$, then either $M=M_{1}\oplus M_{2}$, where
$M_{1}$ and $M_{2}$ are simple modules or M has only one nonzero
proper submodules.
\end{thm}

\begin{proof}
Case 1. Let $x\in Z(R)$ and $x^{2}\neq 0$. It follows that $xM$
and $x^{2}M$ are vertices of $AG(M)$. We have $(xM:M)(xM)\subseteq
x^{2}M$. This implies that $xM\subseteq x^{2}M$ or
$(xM:M)\subseteq (x^{2}M:M)$ so that $xM=x^{2}M$. Let $N$ be a
nonzero proper submodule of $M$ with $N\leq xM$ and let $m\in M$.
Since $xm\in x^{2}M$, there exists $m'\in M$ such that $xm=
x^{2}m'$. Since $M\neq xM$, hence $m-xm'\neq 0$. Thus we have
$x(m-xm')=0\in N$. Again, since $M\neq xM$, $x\in (N:M)$. It
follows that $N=xM$ and hence $xM$ is a minimal submodule of $M$.
By Lemma \ref{l2.4}, we have $N=eM$ for some idempotent $e\in R$.
Now we show that $(1-e)M$ is a minimal submodule of $M$. Let
$0\neq K\subset (1-e)M$. Then there exists $m\in M$ such that
$m(1-e)\notin K$. We have $e(m(1-e))\in K$. So $e\in (K:M)$. It
follows that $e^{2}M=0$, a contradiction. This implies that
$M=M_{1}\oplus M_{2}$, where $M_{1}$ and $M_{2}$ are simple
modules.\\\\
Case 2. Assume that $x^{2}=0$ for every $0\neq x\in Z(R)$. At
first, we show that for every $x, y\in Z(R)\setminus \{0\}$, $xM=
yM$. Otherwise, there exists $m, m'\in M$ such that $xm\notin yM$
and $ym'\notin xM$. We have $x(xm)=0\in yM$ and $y(ym')=0\in xM$.
It follows that $xM\subseteq yM$ and $yM\subseteq xM$, a
contradiction. Hence $xy=0$. It implies that for every vertex $N$
and $K$ of $AG(M)$, $NK=(0)$. Therefore $AG(M)$ is a complete
graph. One can see that for every $0\neq x\in Z(R)$, $xM$ is a
minimal submodule of $M$ and hence there exists $0\neq m\in M$
such that $xM=Rm$. This case and $Z(R)=Nil(R)$ yield that
$Z(R)=Ann(m)$ is a unique prime ideal of $R$. So every nonzero
proper submodule of $M$ is a vertex. Now, if $xM$ is the only
nonzero proper submodule of $M$, then there is nothing to prove.
Otherwise, let $N$ be a nonzero proper submodule of $M$ such that
$xM\subset N$. Thus there exists $m\in N$ such that $m\notin xM$.
If $Ann(m)=(0)$, then $R\cong Rm$. So every nonzero proper ideal
of $R$ is a prime ideal and hence $R$ has only one nonzero proper
ideal. Now the result follows from Theorem \ref{t2.14}. If
$Ann(m)\neq (0)$, then $Rm$ is a minimal submodule of $M$. Since
$xM\subseteq Rm$ ($xM=(xM:M)M=Z(R)M=(Rm:M)M$), we have $xM=Rm$, a
contradiction, as desired.
\end{proof}

\begin{cor}\label{c2.18} Assume that $R$ is not a domain and $M$ is a
 faithful module. Then we have the following.
\begin{itemize}
\item [(a)] $V(AG(M))\subseteq Max(M)$, i.e., every vertex of
$AG(M)$ is a maximal submodule of $M$. \item [(b)] $V(AG(M))=
Max(M)$ \item [(c)] $V(AG(M))=Spec(M)$. \item [(d)]
$V(AG(M))\subseteq Spec(M)$. \item [(e)] Either $M=M_{1}\oplus
M_{2}$, where $M_{1}$ and $M_{2}$ are simple modules or $M$ has
only one nonzero proper submodule.
\end{itemize}

\end{cor}

\begin{proof}
This is clear.
\end{proof}

\section{Coloring of the annihilating-submodule graphs}
We recall that $M$ is an $R$-module.\\
The purpose of this section is to study of coloring of the
annihilating-submodule graphs of modules and investigate the
interplay between $\chi(AG(M))$ and $cl(AG(M))$.\\

\begin{prop}\label{p3.1} Let $M$ be a faithful module.
Then $\chi(AG(M))=1$ if and only if $M$ has only one nonzero
proper submodule.
\end{prop}

\begin{proof}
Let $\chi(AG(M))=1$. Since $AG(M)$ is a connected graph, it can
not have more than one vertex. If $M$ is a faithful module, then
by Corollary \ref{c2.16}(a), $AG(M)$ is a graph with one vertex if
and only if $M$ has only one nonzero proper submodule.
\end{proof}

\begin{thm}\label{t3.2} Let $M$ be a faithful module.
Then the following statements are equivalent.

\begin{itemize}
\item [(a)] $\chi(AG(M)^{*})=2$. \item [(b)] $AG(M)^{*}$ is a
bipartite graph with two nonempty parts. \item [(c)] $AG(M)^{*}$
is a complete bipartite graph with two nonempty parts. \item [(d)]
Either $R$ is a reduced ring with exactly two minimal prime ideals
or $AG(M)^{*}$ is a star graph with more than one vertex.
\end{itemize}

\end{thm}

\begin{proof}
$(a)\Longleftrightarrow (b)$ and $(c)\Longrightarrow (b)$ are
clear.\\
$(b)\Longrightarrow (d)$ Suppose that $AG(M)^{*}$ is a bipartite
graph with two nonempty parts $V_{1}$ and $V_{2}$. One can see
that $AG(M)^{*}$ is a bipartite graph with two nonempty parts
$V_{1}$ and $V_{2}$ if and only if $AG(R)$ is a bipartite graph
with two nonempty parts $U_{1}$ and $U_{2}$ such that if $N\in
V_{i}$, then $(N:M)\in U_{i}$ and if
 $I\in U_{i}$, then $IM\in V_{i}$, for $i=1, 2$.
Hence by \cite[Theorem 2.3]{br12}, $R$ is a reduced ring with
exactly two minimal prime ideals $p_{1}$ and $p_{2}$ or $AG(R)$ is
a star graph with more than one vertex. If $R$ is a reduced ring
with exactly two minimal prime ideals $p_{1}$ and $p_{2}$, then
there is nothing to prove. If $AG(R)$ is a star graph with more
than one vertex,
then $AG(M)^{*}$ is a star graph with more than one vertex.\\
$(d)\Longrightarrow (c)$ Assume that $R$ is a reduced ring with
exactly two minimal prime ideals $p_{1}$ and $p_{2}$. Then by
\cite[Theorem 2.3]{br12}, $AG(R)$ is a complete bipartite graph
with two nonempty parts so that $AG(M)^{*}$ is a complete
bipartite graph with two nonempty parts. If $AG(M)^{*}$ is a star
graph with more than one vertex, then $AG(M)^{*}$ is a complete
bipartite graph.
\end{proof}

\begin{cor}\label{c3.3} Let $R$ be an Artinian ring and assume
 that $M$ is a faithful module.
Then the following statements are equivalent.

\begin{itemize}
\item [(a)] $\chi(AG(M)^{*})=2$. \item [(b)] $AG(M)^{*}$ is a
bipartite graph with two nonempty parts. \item [(c)] $AG(M)^{*}$
is a complete bipartite graph with two nonempty parts. \item [(d)]
Either $M=M_{1}\oplus M_{2}$, where $M_{1}$ and $M_{2}$ are simple
modules or $AG(M)^{*}$ is a star graph with more than one vertex
such that $R$ is a local ring.
\end{itemize}

\end{cor}

\begin{proof}
By Theorem \ref{t3.2}, $(a) \Longleftrightarrow (b)
\Longleftrightarrow (c)$.\\ $(b) \Longrightarrow (d)$ Assume that
$AG(M)^{*}$ is a bipartite graph with two nonempty parts. Hence
$AG(R)$ is a bipartite graph with two nonempty parts. By
\cite[Corollary 2.4]{br12}, if $R\cong F_{1}\oplus F_{2}$, then
$AG(R)$ is a star graph. So $AG(M)^{*}$ is a star graph. Hence by
Corollary \ref{c2.12}, either $M=M_{1}\oplus M_{2}$, where $M_{1}$
and $M_{2}$ are simple modules or $R$ is a local ring with maximal
ideal $p=(0:m)\in Ass(M)$. In the first case, as desired. In the
second case, $AG(M)^{*}$ is a star graph with the center $Rm$. On
the other hand, if $R$ is a local ring such that $AG(R)$ is a star
graph, then we are done.
\end{proof}

\begin{cor}\label{c3.4} Let $R$ be a reduced ring and assume
 that $M$ is a faithful module.
Then the following statements are equivalent.

\begin{itemize}
\item [(a)] $\chi(AG(M)^{*})=2$. \item [(b)] $AG(M)^{*}$ is a
bipartite graph with two nonempty parts. \item [(c)] $AG(M)^{*}$
is a complete bipartite graph with two nonempty parts. \item [(d)]
$R$ has exactly two minimal prime ideals.
\end{itemize}

\end{cor}

\begin{proof}
Use Theorem 3.2.
\end{proof}

Recall that $N< M$ is said to be a semiprime submodule of $M$ if
for every ideal $I$ of $R$ and every submodule $K$ of $M$,
$I^{2}K\subseteq N$ implies that $IK\subseteq N$. Further $M$ is
called a semiprime module if $(0)\subseteq M$ is a semiprime
submodule. Every intersection of prime submodules is a semiprime
submodule (see \cite{tv08}). A prime submodule $N$ of $M$ will be
called extraordinary if whenever $K$ and $L$ are an intersection
of prime submodules of $M$ with $K\cap L\subseteq N$, then
$K\subseteq N$ or $L\subseteq N$ (see \cite{mm97}).

\begin{lem}\label{l3.5} Let $M$ be a semiprime $R$-module such that $AG(M)^{*}$
does not have an infinite clique. Then $M$ has a.c.c. on
submodules of the form $Ann_{M}(I)$, where $I$ is an ideal of $R$.
\end{lem}

\begin{proof}
Suppose that $Ann_{M}(I_{1})\subset Ann_{M}(I_{2})\subset
Ann_{M}(I_{3})\subset...$ (strict inclusions) so that $M$ does not
satisfy the a.c.c. on submodules of the form $Ann_{M}(I)$, where
$I$ is an ideal of $R$. Clearly, $I_{i} Ann_{M}(I_{i+1})\neq 0$,
for each $i\geq 1$. Thus for each $i\geq 1$, there exists
$x_{i}\in I_{i}$ such that $x_{i} Ann_{M}(I_{i+1})\neq (0)$. Let
$J_{i}=x_{i} Ann_{M}(I_{i+1})$, $i=1,2,3,\ldots $ . Then if $i\neq
j$ (we may assume that $i< j$), $J_{i}\neq J_{j}$ because if
$x_{i} Ann_{M}(I_{i+1})= x_{j} Ann_{M}(I_{j+1})$, then we have
$Ann_{M}(I_{i+1})\subseteq Ann_{M}(I_{j})$. So $x_{j}
Ann_{M}(I_{i+1})= (0)$. Hence $x^{2}_{j} Ann_{M}(I_{j+1})= (0)$.
This yields a contradiction because $M$ is a semiprime module and
$x_{j}Ann_{M}(I_{j+1})\neq 0$. On the other hand, one can see that
$J_{i}J_{j}\subseteq x_{i}x_{j} Ann_{M}(I_{i+1})$ and
$J_{i}J_{j}\subseteq x_{i}x_{j} Ann_{M}(I_{j+1})$. Hence we have
$J_{i}J_{j}=(0)$.
\end{proof}

\begin{lem}\label{l3.6} Let $P_{1}=Ann_{M}(x_{1})$ and $P_{2}=Ann_{M}(x_{2})$ be
two distinct elements of $Spec(M)$. Then $(x_{1}M)(x_{2}M)=(0)$.
\end{lem}

\begin{proof}
The proof is straightforward.
\end{proof}

\begin{thm}\label{t3.7} $M$ is a faithful module
if one of the following holds.

\item [(a)] $R$ is a reduced ring and $Z(M)=p_{1}\cup p_{2}\cup
... \cup p_{k}$, where $Min(R)=\{p_{1}, p_{2},..., p_{k}\}$. \item
[(b)] $M$ is a semiprime module and $AG(M)^{*}$ does not have an
infinite clique.
\end{thm}

\begin{proof}
(a). Let $(0)=p_{1}\cap p_{2}\cap...\cap p_{k}$, where $p_{1}$,
$p_{2}$,..., $p_{k}$ are minimal prime ideals of $R$. We have
$Ann(M)\subseteq (p_{1}M:M)\cap ...\cap (p_{k}M:M)$. It is enough
to show that $(p_{i}M:M)=p_{i}$, $i=1,..., k$. For $1\leq i\leq
n$, $p_{1}p_{2}...p_{i-1}p_{i+1}... p_{n}\neq (0)$ because if
$p_{1}p_{2}...p_{i-1}p_{i+1}... p_{n}= (0)$, then $p_{j}\subseteq
p_{i}$, where $j\neq i$, a contradiction. Also, if
$(p_{1}p_{2}...p_{i-1}p_{i+1}... p_{n})M=(0)$, then
$p_{1}p_{2}...p_{i-1}p_{i+1}... p_{n}\subseteq Ann(M)$. So for
every nonzero element $m\in M$, we have
$p_{1}p_{2}...p_{i-1}p_{i+1}... p_{n}\subseteq Ann(m)\subseteq
Z(M)$. It follows that there exists $j\neq i$ such that
$Ann(m)\subseteq p_{j}$. Hence $Z(M)=p_{1}\cup p_{2}\cup ...\cup
p_{i-1}\cup p_{i+1}\cup ... \cup p_{n}$, a contradiction. So
$(p_{1}p_{2}...p_{i-1}p_{i+1}... p_{n})M\neq (0)$. We have
$(p_{i}M)((p_{1}p_{2}...p_{i-1}p_{i+1}... p_{n})M)=0$ and so
$(p_{i}M:M)\subseteq Z(M)$. It follows that $(p_{i}M:M)=p_{i}$,
as desired.\\
(b) Suppose that $M$ is a semiprime module and $AG(M)^{*}$ does
not have an infinite clique. Then by Lemma \ref{l3.5}, $M$ has
a.c.c. on submodules of the form $Ann_{M}(I)$, where $I$ is an
ideal of $R$. Therefore the set $\{Ann_{M}(x)|$ $x\notin Ann(M)\}$
has maximal submodules so that they are prime submodules of $M$.
Let $Ann_{M}(x_{\lambda})$, where $\lambda\in \Lambda$, be the
different maximal members of the family $\{Ann_{M}(x)|$ $x\notin
Ann(M)\}$. By Lemma \ref{l3.6}, the index set $\Lambda$ is finite.
Let $x\in R$ such that $x\notin Ann(M)$. Then $Ann_{M}(x)\subseteq
Ann_{M}(x_{\lambda_{1}})$ for some $\lambda_{1}\in \Lambda$. We
claim that $\cap_{\lambda\in
\Lambda}(Ann_{M}(x_{\lambda}):M)=(0)$. Let $0\neq x\in
\cap_{\lambda\in \Lambda}(Ann_{M}(x_{\lambda}):M)$. So
$xM\subseteq Ann_{M}(x_{\lambda})$ for every $\lambda\in \Lambda$.
We have $Ann_{M}(x)\subseteq Ann_{M}(x_{\lambda_{1}})$. Since
$xM\subseteq Ann_{M}(x_{\lambda_{1}})$, $x_{\lambda_{1}}M\subseteq
Ann_{M}(x)$. Thus $x_{\lambda_{1}}^{2}M=(0)$, a contradiction. Now
the proof is completed because $Ann(M)\subseteq
(Ann_{M}(x_{\lambda}):M)$ for every $\lambda\in \Lambda$.

\end{proof}

\begin{cor}\label{c3.8} Assume that $M$ is a semiprime module. Then
the following statements are equivalent.

\begin{itemize} \item [(a)] $\chi(AG(M)^{*})$ is finite. \item [(b)]
$cl(AG(M)^{*})$ is finite. \item [(c)] $AG(M)^{*}$ does not have
an infinite clique.
\end{itemize}
\end{cor}

\begin{proof}
$(a)\Longrightarrow (b)\Longrightarrow (c)$ is clear.\\
$(c)\Longrightarrow (d)$ Suppose $AG(M)^{*}$ does not have an
infinite clique. It follows directly from the proof of Theorem
\ref{t3.7}(b), there exists a finite number of prime submodules
$P_{1}, ... ,P_{k}$ of $M$ such that $(0)=P_{1}\cap P_{2}\cap ...
\cap P_{k}$. Define a coloring $f(N)=min\{n\in N|$
$(N:M)M\nsubseteq P_{n} \}$ such that $N$ is a vertex of
$AG(M)^{*}$. We have $\chi(AG(M)^{*})\leq k$.
\end{proof}

\begin{cor}\label{c3.9} Assume that $rad_{M}(0)=(0)$ and every
 prime submodule of $M$ is extraordinary. Then
the following statements are equivalent.

\begin{itemize} \item [(a)] $\chi(AG(M)^{*})$ is finite. \item [(b)]
$cl(AG(M)^{*})$ is finite. \item [(c)] $AG(M)^{*}$ does not have
an infinite clique. \item [(d)] $M$ has a finite number of minimal
prime submodules.
\end{itemize}
\end{cor}

\begin{proof}
$(a)\Longrightarrow (b)\Longrightarrow (c)$ is clear.\\
$(c)\Longrightarrow (d)$ Suppose $AG(M)^{*}$ does not have an
infinite clique. Once again, it follows directly from the proof of
Theorem \ref{t3.7}(b), there exists a finite number of prime
submodules $P_{1}, ... ,P_{k}$ of $M$ such that $(0)=P_{1}\cap
P_{2}\cap ... \cap P_{k}$. Since every prime submodule of $M$ is
extraordinary, $M$ has a finite
number of minimal prime submodules.\\
$(d)\Longrightarrow (a)$ Assume that $M$ has a finite number of
minimal prime submodules so that $(0)=P_{1}\cap P_{2}\cap ... \cap
P_{k}$, where $P_{1},...,P_{k}$ are minimal prime submodules of
$M$. Define a coloring $f(N)=min\{n\in N|$ $(N:M)M\nsubseteq P_{n}
\}$ such that $N$ is a vertex of $AG(M)^{*}$. We have
$\chi(AG(M)^{*})\leq k$.
\end{proof}

\begin{lem}\label{l3.10} Let $R$ be a reduced ring and $M$ a faithful $R$-module.
Then $AG(R)$ has an infinite clique if and only if $AG(M)^{*}$ has
an infinite clique.
\end{lem}

\begin{proof}
This is clear.
\end{proof}

\begin{prop}\label{p3.11} Assume that $rad_{M}(0)=(0)$ and $M$ is a faithful $R$-module.
Then the following statements are equivalent.

\begin{itemize} \item [(a)] $\chi(AG(M)^{*})$ is finite. \item [(b)]
$cl(AG(M)^{*})$ is finite. \item [(c)] $AG(M)^{*}$ does not have
an infinite clique. \item [(d)] $R$ has a finite number of minimal
prime ideals.
\end{itemize}
\end{prop}

\begin{proof}
$(a)\Longrightarrow (b)\Longrightarrow (c)$ is clear.\\
$(c)\Longrightarrow (d)$ Suppose $AG(M)^{*}$ does not have an
infinite clique. Then by Theorem \ref{t3.7}(b), $M$ is a faithful
module. Since $rad_{M}(0)=(0)$, it follows that $R$ is a reduced
ring. So by Lemma \ref{l3.5}, $AG(R)$ does not have an infinite
clique. Then by \cite[Corollary 2.10]{br12}, $R$ has a finite
number of minimal prime ideals so that $(0)=p_{1}\cap p_{2}\cap
... \cap p_{k}$, where $p_{1},...,p_{k}$
are prime ideals.\\
$(d)\Longrightarrow (a)$ Assume that $R$ has a finite number of
minimal prime ideals. Since $M$ is a faithful module and
$rad_{M}(0)=(0)$, then $R$ is a reduced ring. So $R$ has a finite
number of minimal prime ideals $p_{1},...,p_{k}$ such that
$(0)=p_{1}\cap p_{2}\cap ... \cap p_{k}$. Define a coloring
$f(N)=min\{n\in N|$ $(N:M)\nsubseteq p_{n} \}$ such that $N$ is a
vertex of $AG(M)^{*}$. We have $\chi(AG(M)^{*})\leq k$.
\end{proof}

\begin{cor}\label{c3.12} Assume that $rad_{M}(0)=(0)$ and $M$ is a faithful module.
Then $\chi(AG(M)^{*})=cl(AG(M)^{*})$. Moreover, if
$\chi(AG(M)^{*})$ is finite, then $R$ has a finite number of
minimal prime ideals, and if $k$ is this number, then
$\chi(AG(M)^{*})=cl(AG(M)^{*})=k$.
\end{cor}

\begin{proof}
Suppose $\chi(AG(M)^{*})$ is finite. Then by Proposition
\ref{p3.11}, $R$ has a finite number of minimal prime ideals
$p_{1},...,p_{k}$. One can see that $R$ is a reduced ring. So
$cl(AG(M)^{*})\leq \chi(AG(M)^{*})\leq k$. By \cite[Theorem
6]{aa12}, $cl(AG(R))\geq k$, and so $cl(AG(M)^{*})\geq k$, as
desired.
\end{proof}

\begin{lem}\label{l3.13}
If $cl(AG(M)^{*})$ is finite, then for every nonzero submodule $N$
of $M$ with $N^{2}=(0)$ and $(N:M)\neq Ann(M)$, $N$ has a finite
number of $R$-submodules $K$ such that $(K:M)\neq Ann(M)$.
\end{lem}

\begin{proof}
This is clear.
\end{proof}

\begin{thm}\label{t3.14} Let $M$ be a Noetherian
module and $\Upsilon=\{N\in V(AG(M)^{*})|$ $N^{2}=(0)\}$. Assume
that every $N\in \Upsilon$ has a finite number of $R$-submodules
in $\Upsilon$. If one of the following statements holds, then
$cl(AG(M)^{*})$ is finite.
\begin{itemize}

\item [(a)] We have $(\Sigma_{N\in \Upsilon} N:M)M=(\Sigma_{N\in
\Upsilon} (N:M)M:M)M $ \item [(b)] For every $N\in \Upsilon$, the
subset $\{K< M|$ $(N:M)M=(K:M)M\}$ is finite.
\end{itemize}

\end{thm}

\begin{proof}
Suppose that every $N\in \Upsilon$ has a finite number of
$R$-submodules in $\Upsilon$ and we have $(a)$. Let $C$ be a
largest clique in $AG(M)^{*}$ and let $\Upsilon_{1}$ be the set of
all vertices $N$ of $C$ with $N^{2}=(0)$. If $\Upsilon_{1} \neq
\emptyset$, then $K=\Sigma_{N\in \Upsilon_{1}} N$ is again a
vertex of $C$ and $K^{2}=(0)$ because for every $L\in C$, we have
$$
(L:M)(\Sigma_{N\in \Upsilon_{1}}N:M)M=(L:M)(\Sigma_{N\in
\Upsilon_{1}}(N:M)M:M)M\subseteq
$$
$$(L:M)(\Sigma_{N\in \Upsilon_{1}}(N:M)M )\subseteq \Sigma_{N\in
\Upsilon_{1}}(L:M)(N:M)M=(0).
$$
Hence $K\in C$. We have
$$
K^{2}=(\Sigma_{N\in \Upsilon_{1}}N:M)^{2}M=(\Sigma_{N\in
\Upsilon_{1}}(N:M)M:M)^{2}M\subseteq ...
$$
$$\subseteq \Sigma_{N, N'\in \Upsilon_{1}}(N:M)(N':M)M=(0).
$$
So by our hypothesis, $K$ has a finite number of $R$-submodules in
$\Upsilon$. But if $N\in \Upsilon_{1}$, every $R$-submodule of $N$
is an $R$-submodule of $K$. Thus for every $N\in \Upsilon_{1}$,
$N$ has a finite number of $R$-submodules in $\Upsilon$ and hence
$\Upsilon_{1}$ has a finite elements. We claim that $C\backslash
\Upsilon_{1}$ has also a finite elements. Suppose that $\{N_{1},
N_{2}, ...\}$ is an infinite subset of $C\backslash \Upsilon_{1}$.
Consider the chain $N_{1}\subseteq N_{1}+N_{2}\subseteq
N_{1}+N_{2}+N_{3}\subseteq \ldots $ . Since $M$ is a Noetherian
module, there exists $n\in N$ such that $N_{1}+... +N_{n}=
N_{1}+... +N_{n}+N_{n+1}$, i.e., $N_{n+1}\subseteq N_{1}+...
+N_{n}$. So
$$
N^{2}_{n+1}\subseteq N_{n+1}(N_{1}+... +N_{n})\subseteq
(N_{n+1}:M)((N_{1}:M)M+...+(N_{n}:M)M)
$$
$$\subseteq
(N_{n+1}:M)(N_{1}:M)M+...+(N_{n+1}:M)(N_{n}:M)M=(0).
$$
It follows that $N^{2}_{n+1}=(0)$, a contradiction. Thus $C$ has a
finite number of vertices and from there, $cl(AG(M)^{*})$ is
finite. Now assume that we have the hypothesis in case $(b)$. Let
$K=\Sigma_{N\in \Upsilon_{1}} (N:M)M$. By using similar arguments
as in case $(a)$, we have $K^{2}=(0)$. But by hypotheses, $K$ has
a finite number of submodules in $\Upsilon$. We claim that
$\Upsilon_{1}$ has a finite number of elements. Suppose not. Then
there exists $N\in \Upsilon_{1}$ such that the subset $\{L\in
\Upsilon | (N:M)M=(L:M)M\}$ is infinite, a contradiction. Thus $C$
has a finite number of vertices and from there, $cl(AG(M)^{*})$ is
a finite set.

\end{proof}

\end{document}